\documentclass[a4paper]{amsart}
\usepackage{amssymb}
\usepackage{amsmath,amscd}
\usepackage{stmaryrd}
\usepackage[all]{xy}
\usepackage{epsf}
\usepackage{enumerate}
\usepackage{combelow}
\newtheorem{theorem}{Theorem}

\newtheorem{proposition}{Proposition}
\newtheorem{lemma}{Lemma}
\newtheorem{corollary}{Corollary}

\theoremstyle{remark}

\newcommand{\rmap}{\longrightarrow}
\newcommand{\diffto}{\xrightarrow{\raisebox{-0.2 em}[0pt][0pt]{\smash{\ensuremath{\sim}}}}}

\begin{document}
\title{Reeb-Thurston stability for symplectic foliations}
\author{Marius Crainic}
\address{Depart. of Math., Utrecht University, 3508 TA Utrecht, The Netherlands}
\email{M.Crainic@uu.nl}
\author{Ioan M\u{a}rcu\cb{t}}
\address{Depart. of Math., Utrecht University, 3508 TA Utrecht, The Netherlands}
\email{I.T.Marcut@uu.nl}
\begin{abstract}
We prove a version the local Reeb-Thurston stability theorem for symplectic foliations. 
\end{abstract}
\maketitle

\section*{Introduction}

A \textbf{symplectic foliation} on a manifold $M$ is a (regular) foliation $\mathcal{F}$, endowed with a 2-form
$\omega$ on $T\mathcal{F}$ whose restriction to each leaf $S$ of $\mathcal{F}$ is a symplectic form
\[\omega_{S}\in\Omega^2(S).\]
Equivalently, a symplectic foliation is a Poisson structure of constant rank.

In this paper we prove a normal form result for symplectic foliations around a leaves. The result uses the
\textbf{cohomological variation} of $\omega$ at the leaf $S$, which is a linear map (see section \ref{Section} for the
definition)
\begin{equation}\label{EQ_coh_var}
[\delta_S\omega]_x:\nu_x^*\rmap H^2(\widetilde{S}_{hol}),\ \ x\in S,
\end{equation}
where $\nu$ denotes the normal bundle of $T\mathcal{F}$, and $\widetilde{S}_{hol}$ is the holonomy cover of $S$.
The cohomological variation arises in fact from a linear map:
\begin{equation}\label{EQ_vari}
\delta_S\omega_x:\nu_x^*\rmap \Omega^2_{\textrm{closed}}(\widetilde{S}_{hol}).
\end{equation}
The local model for the foliation around $S$, which appears in the classical results of Reeb and Thurston, is the flat
bundle $(\widetilde{S}_{hol}\times\nu_x)/\pi_1(S,x)$, where $\pi_1(S,x)$ acts on the second factor via the linear
holonomy
\begin{equation}\label{EQ_lin_hol}
dh:\pi_1(S,x)\rmap Gl(\nu_x).
\end{equation}
For a symplectic foliations the flat bundle can be endowed with leafwise closed 2-forms, which are symplectic in a
neighborhood of $S$; namely, the leaf through $v\in\nu_x$ carries the closed 2-form $j^1_S(\omega)_v$ whose
pull-back to $\widetilde{S}_{hol}\times\{v\}$ is
\begin{equation*}
p^*(j^1_{S}(\omega)_v)=p^*(\omega_S)+\delta_S\omega_x(v).
\end{equation*}

Our main result is the following:
\begin{theorem}\label{Theorem} Let $S$ be an embedded leaf of the symplectic foliation $(M,\mathcal{F},\omega)$. If
the holonomy group of $S$ is finite and the cohomological variation (\ref{EQ_coh_var}) at $S$ is a surjective map, then
some open around $S$ is isomorphic as a symplectic foliation to an open around $S$ in the flat bundle
$(\widetilde{S}_{hol}\times\nu_x)/\pi_1(S,x)$ endowed with the family of closed 2-forms $j^1_S(\omega)$ by a
diffeomorphism which fixes $S$.
\end{theorem}

This result is not a first order normal form theorem, since the holonomy group and the holonomy cover depend on the
germ of the foliation around the leaf. The first order jet of the foliation at $S$ sees only the linear holonomy group
$H_{lin}$ (i.e.\ the image of $dh$) and the corresponding linear holonomy cover denoted $\widetilde{S}_{lin}$. Now,
the map (\ref{EQ_vari}) is in fact the pull-back of a map with values in
$\Omega^2_{\textrm{closed}}(\widetilde{S}_{lin})$. Using this remark, and an extension to noncompact leaves of a
result of Thurston (Lemma \ref{Lemma2}), we obtain the following consequence of Theorem \ref{Theorem}.

\begin{corollary}\label{Corollary1}
Under the assumptions that $S$ is embedded, $\pi_1(S,x)$ is finitely generated, $H_{lin}$ is finite,
$H^1(\widetilde{S}_{lin})=0$ and the cohomological variation
\[[\delta_S\omega]_x:\nu_x^*\rmap H^2(\widetilde{S}_{lin})\]
is surjective, the conclusion of Theorem \ref{Theorem} holds.
\end{corollary}

Our result is clearly related to the normal form theorem for Poisson manifolds around symplectic leaves from
\cite{CrMa}. Both results have the same conclusion, yet the conditions of Theorem \ref{Theorem} are substantially
weaker. More precisely, for regular Poisson manifolds, the hypothesis of the main result in \emph{loc.cit.} are (see
Corollary 4.1.22 and Lemma 4.1.23 \cite{Teza}):
\begin{itemize}
\item the leaf $S$ is compact,
\item $\pi_1(S,x)$ is finite,
\item the cohomological variation is an isomorphism, when viewed as a map
\[[\delta_S\omega]_x:\nu_x^*\rmap H^2(\widetilde{S}_{uni}),\]
where $\widetilde{S}_{uni}$ is the universal cover of $S$.
\end{itemize}
There is yet another essential difference between Theorem 1 and the result from \cite{CrMa}, namely, even in the
setting of Corollary \ref{Corollary1}, the result presented here is a first order result only in the world of symplectic
foliations, and not in that of Poisson structures. The information that a Poisson bivector has constant rank is not
detectable from its first jet.

A weaker version of Theorem \ref{Theorem} is part of the PhD thesis \cite{Teza} of the
second author.

\medskip

\noindent \textbf{Acknowledgments.} This research was financially supported by the ERC Starting Grant no. 279729.

\section{The local model and the cohomological variation}\label{Section}

In this section we describe the local model of a symplectic foliation around a leaf, and define the cohomological
variation of the symplectic structure on the leaves. In the case of general Poisson manifolds, the local model was first
constructed by Vorobjev \cite{Vorobjev}. The approach presented here is more direct; for the relation between these
two constructions see \cite{Teza}.

Let $(M,\mathcal{F})$ be a foliated manifold, and denote its normal bundle by
\[\nu:=TM/T\mathcal{F}.\]
Then $\nu$ carries a flat $T\mathcal{F}$ connection, called the \textbf{Bott connection}, given by
\[\nabla:\Gamma(T\mathcal{F})\times\Gamma(\nu)\rmap \Gamma(\nu),\ \ \nabla_X(\overline{Y}):=\overline{[X,Y]},\]
where, for a vector field $Z$, we denote by $\overline{Z}$ its class in $\Gamma(\nu)$. For a path $\gamma$ inside a
leaf $S$, parallel transport with respect to $\nabla$ gives the \textbf{linear holonomy} transformations:
\[dh(\gamma):\nu_{\gamma(0)}\diffto \nu_{\gamma(1)}.\]
This map depends only on $\gamma$ modulo homotopies inside $S$ with fixed endpoints. Applying $dh$ to closed loops
at $x$, we obtain the \textbf{linear holonomy group}
\[H_{lin,x}:=dh(\pi_1(S,x))\subset Gl(\nu_x).\]
The \textbf{linear holonomy cover} of a leaf $S$ at $x$, denoted by $\widetilde{S}_{lin,x}$ is the covering space
corresponding to the kernel of $dh$; thus it is a principal $H_{lin,x}$ bundle over $S$. Also, $\widetilde{S}_{lin,x}$
can be defined as the space of classes of paths in $S$ starting at $x$, where we identify two such paths if they have the
same endpoint and they induce the same holonomy transport.

The Bott connection induces a foliation $\mathcal{F}_{\nu}$ on $\nu$ whose leaves are the orbits of $dh$; i.e.\ the leaf
of $\mathcal{F}_{\nu}$ through $v\in \nu_x$ covers the leaf $S$ through $x$, and is given by
\[\widetilde{S}_{v}:=\{dh(\gamma)v  :  \gamma \textrm{ is a path in }S\textrm{ starting at }x\}.\]
Therefore, $\widetilde{S}_{lin,x}$ covers of the leaves of the foliation $\mathcal{F}_{\nu}$ above $S$ via the maps
\begin{equation}\label{Covering}
p_v:\widetilde{S}_{lin,x}\rmap \widetilde{S}_v, \ \ p_v([\gamma])=dh(\gamma)v,\ \ v\in \nu_{x}.
\end{equation}

The \textbf{local model} of the foliation around the leaf $S$ is the foliated manifold
\[(\nu_{S},\mathcal{F}_{\nu_S}), \textrm{ where }\mathcal{F}_{\nu_S}:=\mathcal{F}_{\nu}|_{\nu_S}.\]
The linear holonomy induces an isomorphism between the local model and the flat bundle from the Introduction
\[(\widetilde{S}_{lin,x}\times\nu_x)/H_{lin,x}\diffto \nu_S,\ \ [\gamma,v]\mapsto p_v([\gamma]).\]

Consider now a symplectic structure $\omega$ on the foliation $\mathcal{F}$, i.e.\ a 2-form on $T\mathcal{F}$
\[\omega\in\Omega^2(T\mathcal{F})\]
whose restriction to each leaf is symplectic. We first construct a closed foliated 2-form $\delta\omega$ on
$(\nu,\mathcal{F}_{\nu})$, which represents the derivative of $\omega$ in the transversal direction. For this, choose
an extension $\widetilde{\omega}\in \Omega^2(M)$ of $\omega$ and let
\[\Omega(X,Y):=d\widetilde{\omega}(X,Y,\cdot),\ \ X,Y\in T\mathcal{F}.\]
Since $\omega$ is closed along the leaves of $\mathcal{F}$, $\Omega(X,Y)\in\nu^*$, thus $\Omega\in
\Omega^2(T\mathcal{F};\nu^*)$.

Now, the dual of the Bott connection on $\nu^*$ induces a differential $d_{\nabla}$ on the space of foliated forms with
values in the conormal bundle $\Omega^{\bullet}(T\mathcal{F};\nu^*)$; this can be given explicitly by the classical
Koszul formula
\[d_{\nabla}:\Omega^{\bullet}(T\mathcal{F};\nu^*)\rmap \Omega^{\bullet+1}(T\mathcal{F};\nu^*),\]
\begin{align*}
d_{\nabla}\eta(X_0, \ldots , X_{p})=&\sum_{i}(-1)^{i} \nabla_{X_i}\eta(X_0, \ldots , \widehat{X}_i, \ldots , X_{p})+\\
& + \sum_{i< j} (-1)^{i+j}\eta([X_i, X_j],X_0, \ldots , \widehat{X}_i, \ldots, \widehat{X}_j, \ldots , X_{p}),
\end{align*}
for $\eta\in \Omega^{p}(T\mathcal{F};\nu^*)$, $X_i\in\Gamma(T\mathcal{F})$. Denote the resulting cohomology by
$H^{\bullet}(\mathcal{F};\nu^*)$.

It is easy to see that $\Omega$ is $d_{\nabla}$-closed. In fact, this construction can be preformed in all degrees, and it
produces a canonical map (see e.g.\ \cite{CrFe2})
\[d_{\nu}:H^{\bullet}(\mathcal{F})\rmap H^{\bullet}(\mathcal{F};\nu^*),\]
which maps $[\omega]$ to $[\Omega]$. Also, if $\widetilde{\omega}+\alpha$ is a second extension of $\omega$ (where
$\alpha$ vanishes along $\mathcal{F}$), then $\Omega$ changes by $d_{\nabla}\lambda$, where $\lambda\in
\Omega^1(T\mathcal{F};\nu^*)$, is given by \[\lambda(X):=\iota_{X}\alpha\ \  \textrm{ for } \ X\in T\mathcal{F}.\]

Note that there is a natural embedding
\[\mathcal{J}:\Omega^{\bullet}(T\mathcal{F};\nu^*)\rmap \Omega^{\bullet}(T\mathcal{F}_{\nu}),\ \ \mathcal{J}(\eta)_v:=p^*(\langle\eta,v\rangle)|_{T\mathcal{F}_{\nu}},\ \  v\in\nu,\]
where $p:\nu\to M$ is the projection. It is easy to see that under $\mathcal{J}$ the differential $d_{\nabla}$
corresponds to the leafwise de Rham differential $d_{\mathcal{F}_{\nu}}$ on the leaves of $\mathcal{F}_{\nu}$. In
particular, we obtain a closed foliated 2-form
\[\delta\omega:=\mathcal{J}(\Omega)\in \Omega^2(T\mathcal{F}_{\nu}),\]
which we call the \textbf{vertical derivative} of $\omega$. Since $\delta\omega$ vanishes on $M$ (viewed as the zero
section), it follows that $p^*(\omega)+\delta\omega$ is nondegenerate on the leaves in an open around $M$; thus
\[(\nu,\mathcal{F}_{\nu},p^*(\omega)+\delta\omega)\]
is a symplectic foliation around $M$.

Consider now a symplectic leaf $S$. Restricting $p^*(\omega)+\delta\omega$ to the leaves above $S$, we obtain
closed foliated 2-forms along the leaves of the $\mathcal{F}_{\nu_S}$, denoted by
\[j^1_S(\omega):=p^*(\omega_S)+\delta_S\omega\in\Omega^2(T\mathcal{F}_{\nu_S}),\]
where $\omega_S:=\omega|_{S}$ and $\delta_S\omega:=\delta\omega |_{\nu_S}$. Any open neighborhood of $S$ in
\[(\nu_S,\mathcal{F}_{\nu_S},j^1_S(\omega))\]
on which $j^1_S(\omega)$ is symplectic will be regarded as the \textbf{local model} of the symplectic foliation around
$S$; i.e.\ we think about the local model as a germ of a symplectic foliation around $S$.

In order to define the cohomological variation of $\omega$, consider first the linear map
\begin{equation}\label{EQ_variation}
\delta_S\omega_x: \nu_x\rmap \Omega^2_{closed}(\widetilde{S}_{lin,x}), \ \ v\mapsto  p_v^*(\delta_S\omega),
\end{equation}
where the map $p_v$ is the covering map defined by (\ref{Covering}). By the discussion above, choosing a different
extension of $\omega$ changes $p_v^*(\delta_S\omega)$ by an exact 2-form; hence the cohomology class
$[p^*_v(\delta_S\omega)]$ is independent of the 2-form $\Omega$ used to construct $\delta_{S}\omega$. The induced
linear map to the cohomology of $\widetilde{S}_{lin,x}$, will be called the \textbf{cohomological variation} of
$\omega$ at $S$
\[[\delta_S\omega]_x:\nu_{x}\rmap H^2(\widetilde{S}_{lin,x}), \ \ v\mapsto [p^*_v(\delta_S\omega)].\]
In the Introduction we denoted the lifts of $[\delta_S\omega_x]$ to the holonomy cover $\widetilde{S}_{hol}$,
respectively to the universal cover $\widetilde{S}_{uni}$ of $S$, by the same symbol.

We finish this section by proving that, up to isomorphism, the local model is independent of the choices involved. The
proof uses a version of the Moser Lemma for symplectic foliations (Lemma \ref{Lemma5} from the next section).
\begin{proposition}\label{Proposition_independence_of_extension}
Different choices of $\Omega\in \Omega^2(T\mathcal{F},\nu^*)$ satisfying $d_{\nu}[\omega]=[\Omega]$ produce
local models that are isomorphic around $S$ by a diffeomorphism that fixes $S$.
\end{proposition}
\begin{proof}
A second 2-form is of the form
\[\Omega'=\Omega+d_{\nabla}\lambda,\]
for some $\lambda\in \Omega^1(T\mathcal{F};\nu^*)$. We apply the Lemma \ref{Lemma5} to the symplectic foliation
\[(\nu,\mathcal{F}_{\nu},p^*(\omega)+\delta\omega),\]
and the foliated 1-form $\alpha:=\mathcal{J}(\lambda)$ which vanishes along $M$. The resulting diffeomorphism is
foliated. In particular, above any leaf $S$ of $\mathcal{F}$ it sends the local model corresponding to $\Omega$ to the
local model corresponding to $\Omega'$.
\end{proof}

\section{Five lemmas}

In this section we prove some auxiliary results used in the proof of Theorem \ref{Theorem}.

\medskip

\noindent\textbf{Reeb Stability around non-compact leaves}

\medskip

Consider a foliated manifold $(M,\mathcal{F})$ and let $S$ be an embedded leaf. The classical Reeb Stability Theorem
(see e.g.\ \cite{MM}) says that, if the holonomy group $H_{hol}$ is finite and $S$ is compact, then a saturated
neighborhood of $S$ in $M$ is isomorphic as a foliated manifold to the flat bundle
\begin{equation*}
(\widetilde{S}_{hol}\times T)/H_{hol},
\end{equation*}
where $T$ is a small transversal that is invariant under the holonomy action of $H_{hol}$. Since actions of finite
groups can be linearized, it follows that the holonomy of $S$ equals the linear holonomy of $S$. So, some neighborhood
of $S$ in $(M,\mathcal{F})$ is isomorphic as a foliated manifold with the flat bundle from the previous section
\begin{equation}\label{Model_lin}
(\widetilde{S}_{lin}\times \nu_x)/H_{lin}.
\end{equation}

Below we show that the proof of the Reeb Stability Theorem from \cite{MM} can be adapted to the non-compact case,
at the expense of saturation of the open.

\begin{lemma}\label{Lemma1}
Let $(M,\mathcal{F})$ be a foliation and let $S\subset M$ be an embedded leaf. If $S$ has finite holonomy, then an
open neighborhood of $S$ in $M$ is isomorphic as a foliated space to an open around $S$ in the local model
(\ref{Model_lin}), by a diffeomorphism that fixes $S$.
\end{lemma}

\begin{proof}
Since the holonomy is finite, it equals the linear holonomy, and we denote $H:=H_{hol}=H_{lin}$ and
$\widetilde{S}:=\widetilde{S}_{hol}=\widetilde{S}_{lin}$.

The assumption that $S$ be embedded allows us to restrict to a tubular neighborhood; so we assume that the foliation
is on a vector bundle $p:E\to S$ (with $E\cong \nu_{S}$), for which $S$, identified with the zero section, is a leaf. Then
the holonomy of paths in $S$ is represented by germs of a diffeomorphism between the fibers of $E$.

Each point in $S$ has an open neighborhood $U\subset E$ satisfying
\begin{itemize}
\item $S\cap U$ is 1-connected,
\item for $x\in S\cap U$, $E_x\cap U$ is a connected neighborhood of $x$,
\item for every $x,y\in S\cap U$, the holonomy along any path in $S\cap U$ connecting them is defined as a diffeomorphism between the spaces
\[h_{x}^y:E_x\cap U\diffto E_y\cap U.\]
\end{itemize}

Let $\mathfrak{U}$ be locally finite cover of $S$ by opens $U\subset E$ of the type just described, such that for all
$U,U'\in\mathfrak{U}$, $U\cap U'\cap S$ is connected (or empty), and such that each $U\in \mathfrak{U}$ is relatively
compact.

We fix $x_0\in S$, $U_0\in\mathfrak{U}$ an open containing $x_0$, and denote by
\[V:=E_{x_0}.\]
Consider a path $\gamma$ in $S$ starting at $x_0$ and with endpoint $x$. Cover the path by a chain of opens in
$\mathfrak{U}$
\[\xi=(U_{0},\ldots, U_{k(\xi)}),\]
such that there is a partition
\[0=t_0<t_1<\ldots t_{k-1}<t_k=1,\]
with $\gamma([t_{j-1},t_{j}])\subset U_{j}$. Since the holonomy transformations inside $U_{j}$ are all trivial, and all
the intersections $U_i\cap U_j\cap S$ are connected, it follows that the holonomy of $\gamma$ only depends on the
chain $\xi$ and is defined as an embedding
\[h(\gamma)=h_{x_0}^x(\xi):O(\xi)\hookrightarrow E_x,\]
where $O(\xi)\subset V$ is an open neighborhood of $x_0$, which is independent of $x\in U_{k(\xi)}$. Denote by
$\mathcal{Z}$ the space of all chains in $\mathfrak{U}$
\[ \xi=(U_{0}, \ldots, U_{k(\xi)}),\ \  \textrm{with} \ U_{l}\cap U_{l+1}\neq \emptyset.\]

Denote by $K$ the kernel of $\pi_1(S,x_0)\to H$. The holonomy cover $\widetilde{S}\to S$ can be described as the
space of all paths $\gamma$ in $S$ starting at $x_0$, and two such paths $\gamma_1$ and $\gamma_2$ are equivalent
if they have the same endpoint, and the homotopy class of $\gamma_2^{-1}\circ \gamma_1$ lies in $K$. The projection
is then given by $[\gamma]\mapsto \gamma(1)$. Denote by $\widetilde{x}_0$ the point in $\widetilde{S}$
corresponding to the constant path at $x_0$. So, we can represent each point in $\widetilde{S}$ (not uniquely!) by a
pair $(\xi,x)$ with $\xi \in \mathcal{Z}$ and endpoint $x\in U_{k(\xi)}\cap S$.

The group $H$ acts freely on $\widetilde{S}$ by pre-composing paths. For every $g\in H$ fix a chain
$\xi_g\in\mathcal{Z}$, such that $(\xi_g,x_0)$ represents $\widetilde{x}_0g$. Consider the open
\[\widetilde{O}_0:=\bigcap_{g\in H}O(\xi_{g})\subset V,\]
on which all holonomies $h_{x_0}^{x_0}(\xi_{g})$ are defined, and a smaller open $\widetilde{O}_1\subset
\widetilde{O}_0$ around $x_0$, such that $h_{x_0}^{x_0}(\xi_{g})$ maps $\widetilde{O}_1$ into $\widetilde{O}_0$.
Hence the composition
\begin{equation*}
h_{x_0}^{x_0}(\xi_g)\circ h_{x_0}^{x_0}(\xi_h): \widetilde{O}_{1}\hookrightarrow V,
\end{equation*}
is well defined. Since the germs of $h_{x_0}^{x_0}(\xi_g)\circ h_{x_0}^{x_0}(\xi_h)$ and $h_{x_0}^{x_0}(\xi_{gh})$
are the same, by shrinking $\widetilde{O}_1$ if necessary, we may assume that
\begin{equation}\label{composition}
h_{x_0}^{x_0}(\xi_g)\circ h_{x_0}^{x_0} (\xi_h)=h_{x_0}^{x_0}(\xi_{gh}): \widetilde{O}_{1}\hookrightarrow V, \ \ \ \ \forall \ g,h\in H.
\end{equation}
Consider the following open
\[O:=\bigcap_{g\in H}h_{x_0}^{x_0}(\xi_g)(\widetilde{O}_{1}).\]
Then $O\subset \widetilde{O}_{1}$, and for $h\in H$, we have that
\begin{align*}
h_{x_0}^{x_0}(\xi_{h})(O)&\subseteq  \bigcap_{g\in H}h_{x_0}^{x_0}(\xi_{h})\circ h_{x_0}^{x_0}(\xi_{g})(\widetilde{O}_{1})
=\\
&=\bigcap_{g\in H}h_{x_0}^{x_0}(\xi_{hg})(\widetilde{O}_{1})=\bigcap_{g\in H}h_{x_0}^{x_0}(\xi_{g})(\widetilde{O}_{1})= O.
\end{align*}
So $h_{x_0}^{x_0}(\xi_h)$ maps $O$ to $O$, and by (\ref{composition}) it follows that the holonomy transport along
$\xi_g$ defines an action of $H$ on $O$, which we further denote by
\[h(g):=h_{x_0}^{x_0}(\xi_g):O\diffto O.\]
Since $H$ is a finite group acting on $O$ with a fixed point $x_0$, by Bochner's Linearization Theorem, we can
linearize the action around $x_0$. So, by shrinking $O$ if necessary, the action is isomorphic to the linear holonomy
action of $H$ on $V$. In particular, this implies that $O$ contains arbitrarily small $H$-invariant open neighborhoods
of $x_0$.

Since $\mathfrak{U}$ is a locally finite cover by relatively compact opens, there are only finitely many chains in
$\mathcal{Z}$ of a certain length. Denote by $\mathcal{Z}_n$ the set of chains of length at most $n$. Let $c\geq 1$
be such that $\xi_g\in \mathcal{Z}_c$ for all $g\in H$.

By the above, and by the basic properties of holonomy, there exist open neighborhoods $\{O_n\}_{n\geq 1}$ of $x_0$
in $O$:
\[\ldots \subset O_{n+1}\subset O_n\subset O_{n-1}\subset \ldots \subset O_1\subset O\subset V,\]
satisfying the following:
\begin{enumerate}[1)]
\item for every chain $\xi\in \mathcal{Z}_n$, $O_n\subset O({\xi})$,
\item for every two chains $\xi,\xi'\in \mathcal{Z}_n$ and $x\in U_{k(\xi)}\cap U_{k(\xi')}\cap S$, such that the pairs $(\xi,x)$ and
$(\xi',x)$ represent the same element in $\widetilde{S}$, we have that
\[h_{x_0}^x(\xi)=h_{x_0}^x(\xi'):O_n\hookrightarrow E_x,\]
\item $O_n$ is $H$-invariant,
\item for every $g\in H$, $\xi\in\mathcal{Z}_{n}$ and $x\in U_{k(\xi)}\cap S$, we have that
\[h_{x_0}^x(\xi_g\cup \xi)=h_{x_0}^x(\xi)\circ h(g): O_{n+c}\hookrightarrow E_{x}.\]
\end{enumerate}

Denote by $\widetilde{S}_n$ the set of points in $\widetilde{x}\in \widetilde{S}$ for which every element in the orbit
$\widetilde{x}H$ can be represented by a pair $(\xi,x)$ with $\xi\in\mathcal{Z}_n$. Note that for $n\geq c$,
$\widetilde{S}_n$ is nonempty, $H$-invariant, open, and connected. Consider the following $H$-invariant open
neighborhood of $\widetilde{S}\times\{x_0\}$:
\[\mathcal{V}:=\bigcup_{n\geq c}\widetilde{S}_n\times O_{n+c}\subset \widetilde{S}\times V.\]
On $\mathcal{V}$ we define the map
\[\widetilde{\mathcal{H}}:\mathcal{V}\rmap E,  \ \ \ \widetilde{\mathcal{H}}(\widetilde{x},v):=h_{x_0}^x(\xi)(v),\]
for $(\widetilde{x},v)\in\widetilde{S}_n\times O_{n+c}$, where $(\xi,x)$ is pair representing $\widetilde{x}$ with
$\xi\in\mathcal{Z}_{n}$ and $x\in U_{k(\xi)}$. By the properties of the opens $O_n$, $\widetilde{\mathcal{H}}$ is
well defined. Since the holonomy transport is by germs of diffeomorphisms and preserves the foliation, it follows that
$\widetilde{\mathcal{H}}$ is a foliated local diffeomorphism, which sends the trivial foliation on $\mathcal{V}$ with
leaves $\mathcal{V}\cap \big(\widetilde{S}\times\{v\}\big)$ to $\mathcal{F}|_{E}$.

We prove now that $\widetilde{\mathcal{H}}$ is $H$-invariant. Let $(\widetilde{x},v)\in\widetilde{S}_n\times
O_{n+c}$ and $g\in H$. Consider chains $\xi$ and $\xi'$ in $\mathcal{Z}_{n}$ representing $\widetilde{x}$ and
$\widetilde{x}g$ respectively, with $x\in U_{k(\xi)}\cap U_{k(\xi')}\cap S$. Then $\xi'$ and $\xi_g\cup \xi$ both belong
to $\mathcal{Z}_{n+c}$ and $(\xi',x)$, $(\xi_g\cup\xi,x)$ both represent $\widetilde{x}g\in\widetilde{S}$. Using
properties 2) and 4) of the opens $O_n$, we obtain $H$-invariance:
\begin{align*}
\widetilde{\mathcal{H}}(\widetilde{x}g,h(g^{-1})v)&=h_{x_0}^x(\xi')(h(g^{-1})v)=h_{x_0}^x(\xi_g\cup \xi)(h(g^{-1})v)=\\
&=h_{x_0}^x(\xi)\circ h(g)\circ h(g^{-1})v=h_{x_0}^x(\xi)(v)=\widetilde{\mathcal{H}}(\widetilde{x},v).
\end{align*}

Since the action of $H$ on $\mathcal{V}$ is free and preserves the foliation on $\mathcal{V}$, we obtain an induced
local diffeomorphism of foliated manifolds:
\[\mathcal{H}: \mathcal{V}/H\subset (\widetilde{S}\times V)/H \rmap E.\]

We prove now that $\mathcal{H}$ is injective. Let $(\widetilde{x},v), (\widetilde{x}',v')\in \mathcal{V}$ be such that
\[\widetilde{\mathcal{H}}(\widetilde{x},v)=\widetilde{\mathcal{H}}(\widetilde{x}',v').\]
Denoting by $x=p(\widetilde{\mathcal{H}}(\widetilde{x},v))=p(\widetilde{\mathcal{H}}(\widetilde{x}',v'))$, we have
that $\widetilde{\mathcal{H}}(\widetilde{x},v)$, $\widetilde{\mathcal{H}}(\widetilde{x}',v')\in E_x$. Hence
$\widetilde{x}$ and $\widetilde{x}'$, both lie in the fiber of $\widetilde{S}\to S$ over $x$, thus there is a unique $g\in
H$ with $\widetilde{x}'=\widetilde{x}g$. Let $n,m\geq c$ be such that $(\widetilde{x},v)\in \widetilde{S}_n\times
O_{n+c}$ and $(\widetilde{x}',v')\in \widetilde{S}_m\times O_{m+c}$, and assume also that $n\leq m$. Consider
$\xi\in\mathcal{Z}_n$ and $\xi'\in \mathcal{Z}_m$ such that $(\xi,x)$ represents $\widetilde{x}$ and $(\xi',x)$
represents $\widetilde{x}'$. Then we have that
\begin{equation}\label{EQ10}
h_{x_0}^x(\xi)(v)=h_{x_0}^x(\xi')(v').
\end{equation}
Since both $(\xi',x)$ and $(\xi_g\cup \xi,x)$ represent $\widetilde{x}'\in \widetilde{S}$, and both have length $\leq
m+c$, again by the properties 2) and 4) we obtain
\[h_{x_0}^x(\xi')(v')=h_{x_0}^x(\xi_g\cup \xi)(v')=h_{x_0}^x(\xi)(h(g)(v')).\]
Since $h_{x_0}^x(\xi)$ is injective, (\ref{EQ10}) implies that $v=h(g)(v')$. So, we obtain
\[(\widetilde{x},v)=(\widetilde{x}'g^{-1},h(g)(v')),\]
which proves injectivity of $\mathcal{H}$.
\end{proof}

\noindent\textbf{Thurston Stability around non-compact leaves}

\medskip

To obtain the first order normal form result (Corollary \ref{Corollary1}), we will use the following extension to
non-compact leaves of a result of Thurston \cite{Thurston}.

\begin{lemma}\label{Lemma2}
Let $S$ be an embedded leaf of a foliation such that $K_{lin}$, the kernel of $dh:\pi_1(S,x)\to H_{lin}$, is finitely
generated and $H^1(\widetilde{S}_{lin})=0$. Then the holonomy group $H_{hol}$ of $S$ coincides with the linear
holonomy group $H_{lin}$ of $S$.
\end{lemma}
\begin{proof}
Denote by $V:=\nu_x$, the normal space at some $x\in S$. The linear holonomy gives an identification of the normal
bundle of $S$ in $M$ with the vector bundle $(\widetilde{S}_{lin}\times V)/H_{lin}$. Passing to a tubular
neighborhood, we may assume that the foliation $\mathcal{F}$ is on $(\widetilde{S}_{lin}\times V)/H_{lin}$, and that
its linear holonomy coincides with the holonomy of the flat bundle, i.e.\ the first order jet along $S$ of $\mathcal{F}$
equals the first order jet along $S$ of flat bundle foliation. Consider the covering map
\[p:\widetilde{S}_{lin}\times V\rmap (\widetilde{S}_{lin}\times V)/H_{lin}.\]
The leaf $\widetilde{S}_{0}:=\widetilde{S}_{lin}\times\{0\}$ of the pull-back foliation $p^*(\mathcal{F})$ on
$\widetilde{S}_{lin}\times V$ satisfies:
\begin{enumerate}[(1)]
\item $\widetilde{S}_{0}$ has trivial linear holonomy;
\item $H^1(\widetilde{S}_{0})=0$;
\item $\pi_1(\widetilde{S}_{0})\cong K_{lin}$ is finitely generated.
\end{enumerate}
Thurston shows in \cite{Thurston} that, under the assumption that $\widetilde{S}_0$ is compact, the first two
conditions imply that the holonomy group of $\widetilde{S}_0$ vanishes. It is straightforward to check that Thurston's
argument actually doesn't use the compactness assumption, but it only uses condition (3); and we conclude that also in
our case the holonomy at $\widetilde{S}_0$ of $p^*(\mathcal{F})$ vanishes.

Now consider a loop $\gamma$ in $S$ based at $x$ such that $[\gamma]\in K_{lin}$. This is equivalent to saying that
$\gamma$ lifts to a loop in $\widetilde{S}_{lin}$, hence to a loop $\widetilde{\gamma}$ in $\widetilde{S}_0$. The
holonomy transport along $\widetilde{\gamma}$ induced by $p^*(\mathcal{F})$ projects to the holonomy transport of
$\gamma$ induced by $\mathcal{F}$, and since the first is trivial, so is the latter. This proves that $K_{lin}$ is
included in the kernel of $\pi_1(S,x)\to H_{hol}$, and since the other inclusion always holds, we obtain that
$H_{hol}=H_{lin}$.
\end{proof}

\noindent\textbf{Foliated cohomology of products}

\medskip

Let $M$ and $N$ be two manifolds. Consider the product foliation $TM\times N$ on $M\times N$, with leaves
\[M\times\{y\}\subset M\times N,\ \ y\in N.\]
We denote the complex computing the corresponding foliated cohomology by
\[\big(\Omega^{\bullet}(TM\times N),d\big).\]
The elements of $\Omega^{\bullet}(TM\times N)$ can be regarded as smooth families of forms on $M$:
\[\eta=\left\{\eta_y\in \Omega^{\bullet}(M)\right\}_{y\in N}\ \ \textrm{with} \ \ d\eta=\left\{d\eta_y\in \Omega^{\bullet+1}(M)\right\}_{y\in N}.\]
Denote the corresponding cohomology groups by
\[H^{\bullet}(TM\times N).\]

We need two versions of these groups associated to a leaf $M\times\{x\}$, for a fixed $x\in N$. Denote the subcomplex
of foliated forms vanishing on $M\times\{x\}$ by
\[\big(\Omega^{\bullet}_{x}(TM\times N),d\big),\]
and the associated cohomology by
\[H^{\bullet}_{x}(TM\times N).\]
Finally, consider the complex of germs at $M\times\{x\}$ of foliated forms
\[\big(\Omega^\bullet_{\mathrm{g}_{x}}(TM\times N),d\big).\]
This space is the quotient of $\Omega^\bullet(TM\times N)$ by the space of foliated forms that vanish on some open in
$M\times N$ that contains $M\times\{x\}$. The leafwise de Rham differential induces a differential on
$\Omega^\bullet_{\mathrm{g}_{x}}(TM\times N)$. Denote the resulting cohomology by
\[H^{\bullet}_{\mathrm{g}_{x}}(TM\times N).\]
Let also $C^{\infty}_{x}(N)$ denote the space of smooth functions on $N$ vanishing at $x$, and
$C^{\infty}_{\mathrm{g}_{x}}(N)$ denote the space of germs of smooth functions on $N$ around $x$.

These three versions of foliated cohomology come with natural pairings with the homology of $M$, which yield maps:
\begin{align}\label{pairing}
\nonumber &\Psi: H^{\bullet}(TM\times N)\rmap \mathrm{Hom}(H_{\bullet}(M);C^{\infty}(N)),\\
&\Psi_x:H^{\bullet}_{x}(TM\times N)\rmap  \mathrm{Hom}(H_{\bullet}(M);C^{\infty}_{x}(N)),\\
\nonumber &\Psi_{\mathrm{g}_{x}}:H^{\bullet}_{\mathrm{g}_{x}}(TM\times N)\rmap \mathrm{Hom}(H_{\bullet}(M);C^{\infty}_{\mathrm{g}_{x}}(N)).
\end{align}
We explain the third map; the first two are constructed similarly. Consider an element $[\eta]\in
H^q_{\mathrm{g}_{x}}(TM\times N)$, which is represented by a foliated $q$-form $\eta$ that is closed on some open
containing $M\times\{x\}$. We define the corresponding linear map:
\[\Psi_{\mathrm{g}_{x}}([\eta]):H_{q}(M)\rmap C_{\mathrm{g}_{x}}^{\infty}(N).\]
Represent an element $[c]\in H_{q}(M)$ as $c=\sum_i a_i \sigma_i$, where $\sigma_i:\Delta_q\to M$ are smooth
$q$-simplices. Define
\[\langle \eta, c\rangle \in C^{\infty}(N),\ \ y\mapsto \sum_i a_i \int_{\Delta_q}(\sigma_i \times \{y\})^*(\eta).\]
The germ at $x$ of the function $\langle \eta, c\rangle$ is independent of the choice of the representatives, yielding a
well-defined element $\Psi_{\mathrm{g}_{x}}([\eta])([c]):=\langle [\eta],[c]\rangle\in C^{\infty}_{\mathrm{g}_x}(N)$.

\begin{lemma}\label{Lemma3}
The maps from (\ref{pairing}) are linear isomorphisms.
\end{lemma}
\begin{proof}
Denote the constant sheaves on $M$ associated to the groups $C^{\infty}(N)$, $C^{\infty}_x(N)$ and
$C^{\infty}_{\mathrm{g}_x}(N)$ by $\mathcal{S}_1$, $\mathcal{S}_2$ and $\mathcal{S}_3$, respectively.
By standard arguments, the de Rham differential along $M$ induces resolutions $\mathcal{S}_i\to \mathcal{C}_i^{\bullet}$ by fine sheaves on $M$:
\[\mathcal{C}_1^{\bullet}(U):=\Omega^{\bullet}(TU\times N),\  \ \mathcal{C}_2^{\bullet}(U):=\Omega^{\bullet}_x(TU\times N),\ \ \mathcal{C}_3^{\bullet}(U):=\Omega^{\bullet}_{\mathrm{g}_x}(TU\times N).\]
Hence, the foliated cohomologies from (\ref{pairing}) are isomorphic to the sheaf cohomologies with coefficients in $\mathcal{S}_1$, $\mathcal{S}_2$ and $\mathcal{S}_3$ respectively. On the other hand, for any vector space $V$, denoting by $\underline{V}$ the constant sheaf on $M$, one has a natural isomorphism:
\begin{equation}\label{YAE}
\Phi_V:H^{\bullet}(M;\underline{V})\diffto \mathrm{Hom}\big(H_{\bullet}(M);V\big).
\end{equation}
Hence, we obtain isomorphisms:
\begin{align}\label{isose}
\nonumber \Phi: H^{\bullet}(TM\times N)&\diffto \mathrm{Hom}(H_{\bullet}(M);C^{\infty}(N)),\\
\Phi_x:H^{\bullet}_{x}(TM\times N)&\diffto  \mathrm{Hom}(H_{\bullet}(M);C^{\infty}_{x}(N)),\\
\nonumber \Phi_{\mathrm{g}_x}:H^{\bullet}_{\mathrm{g}_{x}}(TM\times N)&\diffto \mathrm{Hom}(H_{\bullet}(M);C^{\infty}_{\mathrm{g}_{x}}(N)).
\end{align}
We still have to check that these maps coincide with those from (\ref{pairing}). For this we will exploit the
naturality of the maps in (\ref{YAE}).

In the first case, consider the evaluation map $ev_y:C^{\infty}(N)\to \mathbb{R}$, for $y\in N$. This induces a sheaf map $ev_y^M:\mathcal{S}_1\to \underline{\mathbb{R}}$ into the constant sheaf over $M$, which is covered by a map $ev_y^M:\mathcal{C}_1^{\bullet}\to \Omega_M^{\bullet}$ into the standard de Rham resolution of $\underline{\mathbb{R}}$. Hence the map $H^{\bullet}(M;\mathcal{S}_1)\to H^{\bullet}(M;\underline{\mathbb{R}})$ induced by $ev_y$ becomes
\[H^{\bullet}(TM\times N)\stackrel{ev_y^M}{\rmap} H^{\bullet}(M), \ \ [\omega]\mapsto [\omega|_{M\times\{y\}}].\]
By naturality of (\ref{YAE}), it follows that the following square commutes:
\[\begin{array}[c]{ccc}
H^{\bullet}(TM\times N)&\stackrel{\Phi}{\rmap} &\mathrm{Hom}\big(H_{\bullet}(M);C^{\infty}(N)\big)\\
\downarrow\scriptstyle{ev_y}&&\downarrow\scriptstyle{ev_y}\\
H^{\bullet}(M)&\stackrel{\Phi_{\mathbb{R}}}{\rmap} &\mathrm{Hom}\big(H_{\bullet}(M);\mathbb{R}\big).
\end{array}\]
Since $\Phi_{\mathbb{R}}$ is the usual isomorphism given by integration, and by the explicit description of the map
$\Psi$, this implies that $\Psi=\Phi$.

For the second map in (\ref{pairing}) and (\ref{isose}) we proceed similarly, but using the inclusion $i:C^{\infty}_x(N)\to C^{\infty}(N)$ instead of $ev_y$. This gives rise to a sheaf map $\mathcal{S}_2\to \mathcal{S}_1$ which lifts to their resolutions, and then we obtain a commutative square
\[\begin{array}[c]{ccc}
H^{\bullet}_x(TM\times N)&\stackrel{\Phi_x}{\rmap} &\mathrm{Hom}\big(H_{\bullet}(M);C^{\infty}_x(N)\big)\\
\downarrow\scriptstyle{i}&&\downarrow\scriptstyle{i}\\
H^{\bullet}(TM\times N)&\stackrel{\Phi}{\rmap} &\mathrm{Hom}\big(H_{\bullet}(M);C^{\infty}(N)\big).
\end{array}\]
Using also that $\Psi=\Phi$, this implies the equality $\Psi_x=\Phi_x$.

Similarly, for the third map in (\ref{pairing}) and (\ref{isose}), but using the projection map $p:C^{\infty}(N)\to C^{\infty}_{\mathrm{g}_x}(N)$ (instead of the inclusion), we obtain a commutative square
\[\begin{array}[c]{ccc}
H^{\bullet}(TM\times N)&\stackrel{\Phi}{\rmap} &\mathrm{Hom}\big(H_{\bullet}(M);C^{\infty}(N)\big)\\
\downarrow\scriptstyle{p}&&\downarrow\scriptstyle{p}\\
H^{\bullet}_{\mathrm{g}_x}(TM\times N)&\stackrel{\Phi_{\mathrm{g}_x}}{\rmap} &\mathrm{Hom}\big(H_{\bullet}(M);C^{\infty}_{\mathrm{g}_x}(N)\big).
\end{array}\]
Again, since $\Psi=\Phi$, we obtain that $\Psi_{\mathrm{g}_x}=\Phi_{\mathrm{g}_x}$. This concludes the proof.
\end{proof}

We will use the following consequences of Lemma \ref{Lemma3} (the first appeared in \cite{GLSW}).
\begin{corollary}\label{Corollary3}
Let $\eta\in \Omega^{q}(TM\times N)$ be a foliated $q$-form such that $\eta_y\in \Omega^{q}(M)$ is exact for all
$y\in N$. Then there exists $\theta\in \Omega^{q-1}(TM\times N)$ such that $d\theta=\eta$. Moreover, if
$\eta_{x}=0$ for some $x\in N$, then one can choose $\theta$ such that $\theta_{x}=0$.
\end{corollary}
\begin{proof}
In the first case, we need that $[\eta]=0$ in $H^{\bullet}(TM\times N)$, and in the second, that $[\eta]=0$ in
$H^{\bullet}_x(TM\times N)$. Since $\langle [\eta_y],[c]\rangle=0$, for all $[c]\in H_{q}(M)$ and all $y\in N$,
the description of the maps $\Psi$ and $\Psi_x$ and Lemma \ref{Lemma3} imply the result.
\end{proof}

\begin{corollary}\label{Corollary4}
Let $\eta$ be a closed foliated $q$-form defined on some open $\mathcal{U}\subset M\times N$ around
$M\times\{x\}$. Then there exists a closed foliated $q$-form $\widetilde{\eta}$ on $M\times N$, such that
$\eta|_{\widetilde{\mathcal{U}}}=\widetilde{\eta}|_{\widetilde{\mathcal{U}}}$, for some open
$\widetilde{\mathcal{U}}\subset \mathcal{U}$ containing $M\times\{x\}$.
\end{corollary}
\begin{proof}
First, we claim that the projection $p:\Omega^{\bullet}(TM\times N)\to \Omega^{\bullet}_{\mathrm{g}_x}(TM\times
N)$ induces a surjective map in cohomology. By the description of the maps $\Psi$ and $\Psi_{\mathrm{g}_x}$, we
have a commutative diagram
\[\begin{array}[c]{ccc}
H^{\bullet}(TM\times N)&\stackrel{\Psi}{\rmap} &\mathrm{Hom}\big(H_{\bullet}(M);C^{\infty}(N)\big)\\
\downarrow\scriptstyle{p}&&\downarrow\scriptstyle{p}\\
H^{\bullet}_{\mathrm{g}_x}(TM\times N)&\stackrel{\Psi_{\mathrm{g}_x}}{\rmap} &\mathrm{Hom}\big(H_{\bullet}(M);C^{\infty}_{\mathrm{g}_x}(N)\big).
\end{array}\]
By Lemma \ref{Lemma3}, the horizontal maps are isomorphisms, and since the vertical map on the right is surjective,
so is the vertical map on the left.

Consider a foliated $q$-form $\eta'\in \Omega^q(TM\times N)$, such that
$\eta'|_{\mathcal{U}'}=\eta|_{\mathcal{U}'}$ for some open $\mathcal{U}'\subset \mathcal{U}$ containing
$M\times\{x\}$. Then $\eta'$ defines a class $[\eta']\in H^q_{\mathrm{g}_{x}}(TM\times N)$. By the above, there is a
closed foliated $q$-form $\eta''\in\Omega^q(TM\times N)$, such that $[\eta'']=[\eta']\in
H^{q}_{\mathrm{g}_x}(TM\times N)$. Thus, there is some foliated $q-1$-form $\theta$ and an open
$\widetilde{\mathcal{U}}\subset \mathcal{U}'$ containing $M\times\{x\}$ such that
$\eta'|_{\widetilde{\mathcal{U}}}=(\eta''+d\theta)|_{\widetilde{\mathcal{U}}}$. The closed foliated $q$-form
$\widetilde{\eta}:=\eta''+d\theta$ satisfies the conclusion:
$\widetilde{\eta}|_{\widetilde{\mathcal{U}}}=\eta|_{\widetilde{\mathcal{U}}}$.
\end{proof}

\noindent\textbf{Equivariant submersions}

\medskip

We prove now that submersions can be equivariantly linearized.

\begin{lemma}\label{Lemma4}
Let $G$ be compact Lie group acting linearly on the vector spaces $V$ and $W$. Consider $f:V\to W$ a smooth
$G$-equivariant map, such that $f(0)=0$. If $f$ is a submersion at $0$, then there exists a $G$-equivariant embedding
$\chi:U\hookrightarrow V$, where $U$ is an invariant open around $0$ in $V$, such that $\chi(0)=0$ and
\[f(\chi(v))=df_0 (v), \  \textrm{ for }\ v\in U.\]
\end{lemma}
\begin{proof}
Since $G$ is compact, we can find a $G$-equivariant projection $p_K:V\to K$, where $K:=\ker(df_0)$. The differential
at $0$ of the $G$-equivariant map
\[(f,p_K):V\rmap W\times K, \ v\mapsto(f(v),p_K(v))\]
is $(df_{0},p_K)$. So $(f,p_K)$ is a diffeomorphism when restricted to some open $U_0$ in $V$ around $0$, which we
may assume to be $G$-invariant. Define the embedding as follows
\[\chi:U\hookrightarrow V, \ \ \ \  \chi:=(f,p_K)^{-1}\circ(df_{0},p_K),\]
where $U:=(df_{0},p_K)^{-1}(U_0)$. Clearly $U$ is $G$-invariant, $\chi$ is $G$-equivariant and $\chi(0)=0$. Since
\[\left(f(\chi(v)),p_K(\chi(v))\right)=\left(df_{0}(v),p_K(v)\right),\]
we also have that $f(\chi(v))=df_{0}(v)$.
\end{proof}

\noindent\textbf{The Moser Lemma for symplectic foliations}

\medskip

The following is a version for symplectic foliations of the Moser Lemma.

\begin{lemma}\label{Lemma5}
Let $(M,\mathcal{F},\omega)$ be a symplectic foliation. Consider a foliated 1-form
\[\alpha\in \Omega^1(T\mathcal{F}),\]
that vanishes on an embedded saturated submanifold $Z$ of $M$. Then $\omega+d_{\mathcal{F}}\alpha$ is
nondegenerate in a neighborhood $U$ of $Z$, and the resulting symplectic foliation
\[(U,\mathcal{F}|_{U},\omega|_{U}+d_{\mathcal{F}}\alpha|_{U})\] is isomorphic around $Z$ to $(M,\mathcal{F},\omega)$ by a
foliated diffeomorphism that fixes $Z$.
\end{lemma}
\begin{proof}
Since $\alpha$ vanishes on $Z$ and $Z$ is saturated, it follows that also $d_{\mathcal{F}}\alpha$ vanishes on $Z$.
Thus, there is an open $V$ around $Z$ such that $\omega+d_{\mathcal{F}}\alpha$ is nondegenerate along the leaves
of $\mathcal{F}|_{V}$. Moreover, by the classical tube lemma from topology, we may choose $V$ such that
\[\omega_t:= \omega+td_{\mathcal{F}}\alpha\in \Omega^2(T\mathcal{F})\]
is nondegenerate along the leaves of $\mathcal{F}|_{V}$, for all $t\in [0,1]$. Consider the time dependent vector field
$X_t$ on $V$, tangent to $\mathcal{F}$, determined by
\[\iota_{X_t}\omega_t=-\alpha, \ \ X_t\in \Gamma(T\mathcal{F}|_{V}).\]
Since $X_t$ vanishes along $Z$, again by the tube lemma, there is an open $O\subset V$ around $Z$, such that the
flow $\Phi^t_{X}$ of $X_t$ is defined up to time 1 on $O$. We claim that $\Phi^1_X$ gives the desired isomorphism.
Clearly $\Phi^1_X$ preserves the foliation and is the identity on $Z$. On each leaf $S$, we have that
\begin{align*}
\frac{d}{dt}\Phi_X^{t*}(\omega_{t}|_S)=\Phi_X^{t*}(L_{X_t}\omega_{t}|_S+d_{\mathcal{F}}\alpha|_{S})=\Phi_X^{t*}(d\iota_{X_t}\omega_{t}|_S+d\alpha|_{S})=0.
\end{align*}
Thus $\Phi_X^{t*}(\omega_{t}|_S)$ is constant, and since $\Phi_X^0=\textrm{Id}$, we have that
\[\Phi_X^{1*}\left((\omega+d_{\mathcal{F}}\alpha)|_S\right)=\omega|_{S}.\]
So, $\Phi^1_X$ is an isomorphism between the symplectic foliations
\[\Phi_X^1:(O,\mathcal{F}|_{O},\omega|_{O})\diffto (U,\mathcal{F}|_{U},\omega|_{U}+d_{\mathcal{F}}\alpha|_{U}),\]
where $U:=\Phi^1_X(O)$.
\end{proof}

\section{Proof of Theorem 1}

Since the holonomy of $S$ is finite, it coincides with the linear holonomy. Consider $x\in S$ and denote by
$V:=\nu_{x}$, by $H:=H_{hol}=H_{lin}$, and by $\widetilde{S}:=\widetilde{S}_{hol}=\widetilde{S}_{lin}$.
Applying Lemma \ref{Lemma1}, we obtain that some neighborhood of $S$ in $M$ is diffeomorphic as a foliated
manifold to an open $\mathcal{U}$ around $S$ in the flat bundle
\[(\widetilde{S}\times V)/H.\]
The symplectic leaves correspond to the connected components of $S_v\cap \mathcal{U}$, where
\[S_v:=(\widetilde{S}\times Hv)/H,\ \ v\in V.\]
We claim that there exists $\omega_1$, a closed foliated 2-form on $(\widetilde{S}\times V)/H$ that extends
$\omega|_{\mathcal{U}_1}$, for some open $\mathcal{U}_1\subset \mathcal{U}$ around $S$. For this, consider the
projection
\[p:\widetilde{S}\times V\rmap (\widetilde{S}\times V)/H,\]
and denote by $\widetilde{\mathcal{U}}:=p^{-1}(\mathcal{U})$ and by $\widetilde{\omega}:=p^*(\omega)$, which
is a closed foliated 2-form on the product foliation restricted to $\widetilde{\mathcal{U}}$. By Corollary
\ref{Corollary4}, there is a closed extension $\widetilde{\omega}_0$ of
$\widetilde{\omega}|_{\widetilde{\mathcal{U}}_0}$, where $\widetilde{\mathcal{U}}_0\subset
\widetilde{\mathcal{U}}$ is an open around $\widetilde{S}\times \{0\}$. Define $\widetilde{\omega}_1$ by averaging
over $H$
\[\widetilde{\omega}_1:=\frac{1}{|H|}\sum_{g\in H}g^*(\widetilde{\omega}_0).\]
Since $\widetilde{\omega}$ is invariant, it follows that $\widetilde{\omega}_1$ coincides with $\widetilde{\omega}$
on $\widetilde{\mathcal{U}}_1:=\bigcap_{g\in H} g\widetilde{\mathcal{U}}_0$. Since $\widetilde{\omega}_1$ is
invariant, it is of the form $\widetilde{\omega}_1=p^*(\omega_1)$, where $\omega_1$ is a closed foliated 2-form on
$(\widetilde{S}\times V)/H$, which extends the restriction to $\mathcal{U}_1:=p(\widetilde{\mathcal{U}}_1)$ of
$\omega$.

We will identify foliated $q$-forms $\eta$ on $(\widetilde{S}\times V)/H$, with smooth $H$-equivariant families
$\{\eta_v\in \Omega^q(\widetilde{S})\}_{v\in V}$, where $\eta_v:=p^*(\eta)|_{\widetilde{S}\times\{v\}}$.

We compute now the variation of $\omega$ at $S$. Since $\omega$ and $\omega_1$ coincide around $S$, they have the
same variation at $S$. Using the extension of $\omega$ (or equivalently of $\omega_1$) that vanishes on vectors
tangent to the fibers of the projection to $S$, we see that the variation $\delta_S\omega$ is given by the
$H$-equivariant family:
\[\delta_{S}\omega_v:=\frac{d}{d\epsilon}\omega_{ \epsilon v}|_{\epsilon=0}\in \Omega^2(\widetilde{S}),\]
The local model is represented by the $H$-equivariant family of 2-forms:
\[j^1_{S}(\omega)_v=p^*(\omega_S)+\delta_S\omega_v\in \Omega^2(\widetilde{S}).\]

Consider the $H$-equivariant map
\[f:V\rmap H^2(\widetilde{S}),\ \ f(v)=[\omega_{1,v}]-[p^*(\omega_S)].\]
Smoothness of $f$ follows from Lemma \ref{Lemma3}. Clearly, $f(0)=0$ and its differential at $0$ is the
cohomological variation
\[df_{0}(v)=[\delta_{S}\omega]v, \ \ \forall \ v\in V.\]
By our hypothesis, $f$ is a submersion at $0$. So we can apply Lemma \ref{Lemma4} to find an $H$-equivariant
embedding
\[\chi:U\hookrightarrow V,\]
where $U$ is an $H$-invariant open neighborhoods of $0$ in $V$, such that
\[\chi(0)=0 \ \textrm{ and }\ f(\chi(v))=df_0(v).\]

By $H$-equivariance, $\chi$ induces a foliation preserving embedding
\[\widetilde{\chi}: (\widetilde{S}\times U)/H \hookrightarrow (\widetilde{S}\times V)/H, \ \ \widetilde{\chi}([y,v])=[y,\chi(v)]\]
that restricts to a diffeomorphism between the leaf $S_v$ and the leaf $S_{\chi(v)}$. The pullback of $\omega_1$ under
$\widetilde{\chi}$ is the $H$-equivariant family
\[\omega_2=\{\omega_{2, v}:=\omega_{1, \chi(v)}\}_{v\in U}.\]

We have that
\[[\omega_{2,v}]-[p^*(\omega_S)]=[\omega_{1,\chi(v)}]-[p^*(\omega_S)]=f(\chi(v))=df_0(v)=[\delta_S\omega]v.\]

Equivalently, this relation can be rewritten as
\[\omega_{2,v}=j^1_{S}(\omega)_v+\alpha_v, \forall v\in U,\]
where $\{\alpha_v\}_{v\in U}$ is an $H$-equivariant family of exact 2-forms that vanishes for $v=0$. By Corollary
\ref{Corollary3}, $p^*(\alpha)$ is an exact foliated form on $\widetilde{S}\times U$, and moreover, we can choose a
primitive $\widetilde{\beta}\in\Omega^1(T\widetilde{S}\times U)$ such that $\widetilde{\beta}_0=0$. By averaging,
we may also assume that $\widetilde{\beta}$ is $H$-equivariant, thus it is of the form $\widetilde{\beta}=p^*(\beta)$
for a foliated 1-form on $\beta$ on $(\widetilde{S}\times U)/H$ that vanishes along $S$. We obtain:
\[\omega_2=j^1_S\omega+d \beta.\]
Applying Lemma \ref{Lemma5}, we conclude that, on some open around $S$, $j^1_S\omega$ and $\omega_2$ are
related by a foliated diffeomorphism. Now, $\omega_2$ and $\omega_1$ are related by $\widetilde{\chi}$, and
$\omega_1$ and $\omega$ have the same germ around $S$. This concludes the proof.\\

\noindent\textbf{Proof of Corollary \ref{Corollary1}}

\medskip

Schreier's Lemma says that a subgroup of finite index of a finitely generated group is also finitely generated (see e.g.\
section 5.6 in \cite{Rotman}). Hence, $K_{lin}$ is finitely generated. By Lemma \ref{Lemma2}, $H_{hol}=H_{lin}$,
in particular $H_{hol}$ is finite, and so we are in the setting of Theorem \ref{Theorem}.

\bibliographystyle{amsplain}

\begin{thebibliography}{11}

\bibitem{Tu} R.~Bott, L.~Tu, \emph{Differential forms in algebraic topology}, Graduate Texts in Mathematics, 82,  Springer-Verlag, New York-Berlin, 1982.

\bibitem{CrFe2} M.~Crainic, R.L.~Fernandes, Integrability of Poisson brackets, \emph{J.~Differential Geom.}~\textbf{66} (2004), 71--137.

\bibitem{CrMa} M.~Crainic, I.~M\u{a}rcu\cb{t}, A normal form theorem around symplectic leaves, \emph{J. Differential Geom.} {\bf 92} (2012), no. 3, 417--461.

\bibitem{GLSW}M.J.~Gotay, R.~Lashof, J.~\'Sniatycki, A.~Weinstein, Closed forms on symplectic fibre bundles, \emph{Commentarii Mathematici Helvetici}, \textbf{58} (1983), Issue 1, 617--621.

\bibitem{Teza} I.~M\u{a}rcu\cb{t}, \emph{Normal forms in Poisson geometry}, PhD thesis, Utrecht University 2013, arXiv:1301.4571.

\bibitem{MM} I.~Moerdijk and J.~Mr\v{c}un, \emph{Introduction to foliations and Lie groupoids}, Cambridge Studies in Advanced Mathematics, 91. Cambridge University Press, 2003.

\bibitem{Rotman} J.~Rotman, \emph{Advanced modern algebra}, Prentice Hall, Inc., Upper Saddle River, NJ, 2002.

\bibitem{Thurston} W.~Thurston, A generalization of the Reeb stability theorem, \emph{Topology} \textbf{13} (1974), 347--352.

\bibitem{Vorobjev} Y.~Vorobjev, Coupling tensors and Poisson geometry near a single symplectic leaf, \emph{Banach Center Publ.~}\textbf{54} (2001), 249--274.


\end{thebibliography}
\def\lllll{}

\end{document}